\documentclass{amsart}

\usepackage{amsfonts, amssymb}

\usepackage{amssymb,amsmath,amscd,euscript,verbatim,array}
\usepackage{geometry}
\usepackage{anysize}
\usepackage{fancyhdr}
\usepackage{indentfirst}
\usepackage{graphicx}
\usepackage{color}
\usepackage{ifpdf}
\usepackage{marginnote}
\usepackage{enumitem}
\usepackage{mathrsfs}
\usepackage{hyperref}
\usepackage{CJK}
\usepackage{mathtools}
\usepackage{extarrows}

\setlist[itemize]{leftmargin=2em}

\newcommand{\N}{\mathbb{N}}

\newcommand{\R}{\mathbb{R}}

\newcommand{\ka}{\kappa}

\newcommand{\al}{\alpha}

\newcommand{\eps}{\varepsilon}

\newcommand{\T}{\mathbb{T}}

\newtheorem{Theorem}{Theorem}
\newtheorem{Definition}{Definition}[section]
\newtheorem{Proposition}[Definition]{Proposition}

\newtheorem{Lemma}[Definition]{Lemma}

\theoremstyle{definition}
\newtheorem{Remark}[Definition]{Remark}

\newcommand{\bthm}{\begin{Theorem}}
	\newcommand{\ethm}{\end{Theorem}}
\newcommand{\bpr}{\begin{Proposition}}
	\newcommand{\epr}{\end{Proposition}}
\newcommand{\blm}{\begin{Lemma}}
	\newcommand{\elm}{\end{Lemma}}
\newcommand{\bex}{\begin{Exercise}}
	\newcommand{\eex}{\end{Exercise}}
\newcommand{\be}{\begin{equation}}
	\newcommand{\ee}{\end{equation}}
\newcommand{\beal}{\begin{aligned}}
	\newcommand{\enal}{\end{aligned}}
\newcommand{\brm}{\begin{Remark}}
	\newcommand{\erm}{\end{Remark}}

\begin{document}
	\begin{CJK*}{GBK}{kai}
		
		\title[Non-KAM Invariant Circles for Area-Preserving Twist Maps]{  On Non-KAM Invariant Circles for Area-Preserving Twist Maps}
		
        \author{Jiashen Guo}
		\address{School of Mathematics and Statistics, Beijing Institute of Technology, Beijing 100081, China}
		\email{guojiashen775@gmail.com}

		\author{Yi Liu}
		\address{School of Mathematics and Statistics, Beijing Institute of Technology, Beijing 100081, China}
		\email{yiliu111@foxmail.com}
		
		\author{Lin Wang}
		\address{School of Mathematics and Statistics, Beijing Institute of Technology, Beijing 100081, China}
		\email{lwang@bit.edu.cn}

		\subjclass[2020]{37J40,  37E40}
		
		\keywords{Area-Preserving Twist Maps, Arnold's Family, Invariant Circles}

		\begin{abstract}
					 In this note, we investigate the dynamics of invariant circles in area-preserving twist maps. The invariant circles under consideration lie beyond the applicability of classical KAM theory, as the perturbations involved exceed the scope of standard KAM methods. By integrating both constructive and non-constructive techniques, we establish several results on the existence and breakdown of such non-KAM invariant circles. These findings provide new evidence in support of affirmative answers to two open questions posed by Mather in 1998.
		\end{abstract}

		\maketitle

		\tableofcontents
		%\newpage
		\section{\sc Introduction}

		The existence of invariant circles (i.e., homotopically non-trivial invariant curves) for area-preserving twist maps was first established by Moser in his seminal work \cite{Mo1}, commonly referred to as \emph{Moser's twist theorem}. This groundbreaking result is also recognized as one of the foundational contributions to the KAM (Kolmogorov-Arnold-Moser) theory, where the map was initially required to be of class \( C^{333} \). The regularity condition was significantly relaxed to \( C^{4+\eps} \) by R\"{u}ssmann \cite{R1}. Subsequently, Moser \cite[Page 53]{Mo2} asserted that \( C^{3+\eps} \) regularity suffices for the existence of invariant circles with constant type rotation numbers by employing more refined estimates within the KAM framework. The full proof of this assertion was independently provided by R\"{u}ssmann \cite{R2} and Herman \cite{H1}, each utilizing distinct methodologies. In \cite{R2}, R\"{u}ssmann introduced a novel iteration process that yields invariant circles which are merely continuous. In contrast, Herman's approach in \cite{H1} leverages the Schauder fixed point theorem, building upon his celebrated work on circle diffeomorphisms (\cite{H11}). Furthermore, Herman \cite{H33}  showed that an invariant circle with a constant type rotation number persists under small perturbations in the $C^{3}$ topology. By a geometrical approach, Herman \cite{H1}   constructed an example demonstrating that each invariant circle can be destroyed by a small $C^{3-\eps}$ perturbation. This implies that $C^{3}$ regularity is optimal.
		
		The preceding analysis reveals that for a given rotation numbers, one cannot generally expect invariant circles to persist under $C^r$-small perturbations when $r < 3$. Such persistence questions lie beyond the reach of classical KAM theory and its modern extensions. Indeed, even under favorable conditions, constructing explicit examples where such non-KAM invariant circles either persist or are destroyed remains a  challenge. We begin by formalizing the central object of study:

\begin{Definition}
A \emph{non-KAM invariant circle} is an invariant circle of an integrable area-preserving twist map that persists under $C^r$-small perturbations for some $r < 3$, but  $C^3$-large.
\end{Definition}

As a preliminary approach to identifying such invariant circles, we assume for simplicity that the perturbations themselves are either $C^\infty$ or real-analytic ($C^\omega$). By the celebrated Birkhoff graph theorem~\cite{B20} (see also \cite[Chapter I]{H1}), if the area-preserving twist map is of class $C^1$, then any invariant circle must be a Lipschitz graph. Therefore, in what follows, we refer to \emph{invariant circles} as \emph{invariant graphs}.

Under the above definition, we focus primarily on the following questions:
\begin{enumerate}
\item[(I)] \textbf{Existence:} \textit{Can one construct examples of area-preserving twist maps that admit such non-KAM invariant graphs?}

\item[(II)] \textbf{Dynamics:} \textit{What is the dynamics of the map restricted to such non-KAM invariant graphs? In particular, are there constraints on the rotation number as a circle homeomorphism?}

\item[(III)] \textbf{Fragility:} \textit{How fragile are these invariant graphs? For instance, can they be destroyed by $C^\infty$-small perturbations?}

\item[(IV)] \textbf{Smoothness:} \textit{How pathological can the regularity of these invariant graphs as functions be?}
\end{enumerate}

		Herman and Mather made fundamental contributions to the development of twist map theory during the 1980s and 1990s. With the aid of the powerful tools they developed, the study of area-preserving twist maps has achieved significant progress. However, certain questions of a similar nature, which remain not fully resolved to this day, were explicitly raised by both of them and others in various contexts; see, for example, \cite{FKO, H98, MY}. In this note, we aim to contribute to the study of these questions by refining and extending the tools introduced by Herman and Mather.
	\vspace{1ex}	
\subsection{Basic Setting}		
	To state the main results, let us recall some notions and notations. Denote \( \mathbb{T} := \mathbb{R}/\mathbb{Z} \) the flat circle. Let \( |\cdot| \) denote the Euclidean norm on \( \mathbb{R} \). The induced flat metric on \( \mathbb{T} \) is defined as
		\[
		\|\cdot\| := \inf_{p \in \mathbb{Z}} |\cdot + p|.
		\]
		Let \( \alpha \) be an irrational number, expressed via its continued fraction expansion as \( \alpha = [a_0, a_1, \ldots] \). The convergents \( \left\{ {p_n}/{q_n} \right\}_{n \in \mathbb{N}} \) of \( \alpha \) are defined by \( {p_n}/{q_n} = [a_0, \ldots, a_n] \). Without loss of generality, we assume \( \alpha \in [0, 1) \) and \( q_n > 0 \) for all \( n \in \mathbb{N} \).
		
		An irrational number \( \alpha \) is said to be of \emph{constant type} if the coefficients \( a_i \) in its continued fraction expansion are uniformly bounded. We write \( \mathcal{C} \) for the set of constant type irrational numbers in \( (0, 1) \). It is well known that \( \mathcal{C} \) has zero Lebesgue measure but full Hausdorff dimension. More generally, we say that \( \alpha \) is a Diophantine number if there exist constants \( D > 0 \) and \( \tau \geq 0 \) such that for all integers \( p, q \) with \( q \neq 0 \),
\begin{equation}\label{diop}
\left|\alpha - \frac{p}{q}\right| \geq \frac{D}{|q|^{2+\tau}}.
\end{equation}
To emphasize the dependence on the parameters \( D \) and \( \tau \), we also refer to such \( \alpha \) as \((D, \tau)\)-\emph{Diophantine numbers}. Hence, any number of constant type is a \((D, 0)\)-Diophantine number, and vice versa. Irrational numbers that are not Diophantine are called \emph{Liouville numbers}. The set of Brjuno numbers is a larger class of irrationals than the Diophantine numbers. The Diophantine condition can be expressed in terms of this expansion as
\begin{equation}\label{dionu}
\ln q_{n+1} \leq C \ln q_n, \quad n \in \mathbb{N}.
\end{equation}
The Brjuno condition can be stated as
\begin{equation}\label{brun}
\sum_{n \geq 0} \frac{\ln q_{n+1}}{q_n} < +\infty.
\end{equation}

We denote by \( \mathrm{Diff}_+^r(\mathbb{R}) \) (resp. \( \mathrm{Diff}_+^r(\mathbb{T}) \)) the group of \( C^r \) orientation-preserving diffeomorphisms on \( \mathbb{R} \) (resp. \( \mathbb{T} \)), where \( 0 \leq r \leq \omega \). Define
		\[
		D^r(\mathbb{T}) := \left\{ f \in \mathrm{Diff}_+^r(\mathbb{R})\ \middle|\ f - \mathrm{Id} \in C^r(\mathbb{T}) \right\},
		\]
		which is the universal covering space of \( \mathrm{Diff}_+^r(\mathbb{T}) \). For \( f \in D^r(\mathbb{T}) \), the rotation number \( \rho(f) \) is well defined. Define the following sets:
		\[
		\mathcal{F}_\al^r := \left\{ f \in D^r(\mathbb{T})\ \middle|\ \rho(f) = \alpha \right\}, \quad
		\mathcal{O}_\al^r := \left\{ h^{-1} \circ R_\alpha \circ h\ \middle|\ h \in D^r(\mathbb{T}) \right\},
		\]
		\[
		{H}^r := \left\{ \phi \in C^r(\mathbb{T})\ \middle|\ \int_{\mathbb{T}} \phi(\theta)\, d\theta = 0 \right\}.
		\]
		A classical result by Herman and Yoccoz \cite{H11,Y1} shows that \( \mathcal{F}_\al^\infty = \mathcal{O}_\al^\infty \) when \( \alpha \) is a Diophantine number.

		Given $\beta \in \mathbb{R}$, we consider the family of maps \( f_\beta^{\phi} : \mathbb{T} \times \mathbb{R} \to \mathbb{T} \times \mathbb{R} \) defined by
\begin{equation}\label{Fpar}
f_\beta^{\phi}(x, y) = \bigl(x + \beta + y + \phi(x),\ y + \phi(x)\bigr),
\end{equation}
where \( \phi \) is a \(1\)-periodic smooth function with zero mean: \( \int_0^1 \phi(x)\, dx = 0 \). Then \( f_\beta^{\phi} \) is an exact area-preserving twist map. In particular, \( f_\beta^0(x,y) = (x + \beta + y,\, y) \) is integrable. Let \( F_\beta^{\phi} : \mathbb{R} \times \mathbb{R} \to \mathbb{R} \times \mathbb{R} \) denote the lift of \( f_\beta^{\phi} \). If it exists, we denote an invariant graph of \( f_\beta^{\phi} \) by
\[
\Gamma := \{(x, \Psi(x)) \mid x \in \mathbb{T}\}.
\]
For simplicity, we use the same notations for the following objects and their lifts: \( \phi, \Gamma, \Psi, g \).

\subsection{Related work by Herman}

A key observation of Herman~\cite{H1} (also noted by Mather \cite{M1}) is that \( F_\beta^{\phi} \) admits an invariant graph \( \Psi \) over \( \mathbb{T} \) if and only if there exists \( g_\beta \in D^r(\mathbb{T}) \) such that
\begin{equation}\label{hrfo}
\frac{g_\beta + g_\beta^{-1}}{2} = \mathrm{Id} + \frac{1}{2} \phi.
\end{equation}
Moreover, the circle diffeomorphism \( g_\beta \) can be written as
\[
g_\beta(x) = x + \beta + \Psi(x) + \phi(x), \quad x \in \mathbb{R}.
\]
For \( 0 \leq r \leq \omega \), define the map \( \Phi: \mathcal{F}_\alpha^r \to {H}^r \) by
\[
\Phi(g) := \frac{g + g^{-1}}{2} - \mathrm{Id}.
\]

\begin{Remark}
Herman~\cite{H1} originally established formula (\ref{hrfo}) for maps of the form
\[
F(x,y) = (x + y,\, y + \phi(x + y)).
\]
Through a straightforward modification (see Lemma~\ref{herob}), one can show that \eqref{hrfo} remains valid for the parameter-dependent family \( F_\beta^{\phi} \), and the map \( \Phi \) is independent of \( \beta \); that is, for any \( \beta_1, \beta_2 \in \mathbb{R} \),
\[
\Phi(g_{\beta_1}) = \Phi(g_{\beta_2}).
\]
This elementary fact plays a crucial role in the proofs of Theorems~\ref{M0} and~\ref{M01}.
	\end{Remark}
		\begin{Remark}
			Let \( \mathrm{Lip}_+(\mathbb{R}) \) denote the group of bi-Lipschitz orientation-preserving maps on \( \mathbb{R} \). Define
			\[
			D^{\mathrm{Lip}}(\mathbb{T}) := \left\{ f \in \mathrm{Lip}_+(\mathbb{R})\ \middle|\ f - \mathrm{Id} \text{ is Lipschitz on } \mathbb{T} \right\},
			\]
			\[
			\mathcal{F}_\al^{\mathrm{Lip}} := \left\{ f \in D^{\mathrm{Lip}}(\mathbb{T})\ \middle|\ \rho(f) = \alpha \right\}.
			\]
			The Birkhoff graph theorem implies $\Phi(\mathcal{F}_\al^0) = \Phi(\mathcal{F}_\al^{\mathrm{Lip}})$.
			\end{Remark}	
		
	By Herman~\cite[Chapter II]{H1}, we know
\begin{Proposition}\label{Hp123}
\
\begin{itemize}
\item [(i)] \( \Phi(\mathcal{F}_\al^\infty) \) is open in \( {H}^\infty \).
\item [(ii)] $\Phi(\mathcal{F}_\al^0)\cap H^\infty$ is closed in $H^\infty$ under the $C^1$ topology.
\item [(iii)] \( \Phi \) is injective for irrational \( \alpha \).
\end{itemize}
\end{Proposition}	
		Let
		\[
		U^r_\delta := \left\{ \phi \in C^\infty(\mathbb{T})\ \middle|\ \|\phi\|_{C^{r}} < \delta \right\}.
		\]
		By \cite[Corollary 7.10]{H33}, we know that for each $\alpha \in \mathcal{C}$, there exists $\delta > 0$ such that
		\begin{equation}\label{hkam}
			U_\delta^3 \subseteq \Phi(\mathcal{F}_\alpha^\infty).
		\end{equation}
		On the other hand, Herman \cite[Theorem 3.3]{H1} constructed a $C^\infty$ perturbation $\phi$ that is arbitrarily small in the $C^{2 - \varepsilon}$ topology, such that all invariant graphs of the  twist map $F^\phi$ are destroyed. Let
		\[
		\mathcal{F}^0 := \bigcup_{\alpha \in \mathbb{R}} \mathcal{F}_\alpha^0.
		\]
		Then Herman's result can be reformulated as follows: for each $\delta > 0$, we have
		\begin{equation}\label{hckam1}
			(U_\delta^{2 - \varepsilon} \setminus U_\delta^2) \cap \Phi(\mathcal{F}^0) \subsetneqq U_\delta^{2 - \varepsilon} \setminus U_\delta^2.
		\end{equation}
		%To the best of our knowledge, there is no result in the literature proving that
%		\[
%		(U_\delta^{2 - \varepsilon} \setminus U_\delta^2) \cap \Phi(\mathcal{F}^0) \neq \emptyset.
%		\]
		A natural question concerns the structure of the set $(U_\delta^{2 - \varepsilon} \setminus U_\delta^2) \cap \Phi(\mathcal{F}^0)$. In \cite[Chapter III]{H1}, Herman proved  that
\begin{Proposition}\label{Hpiii}
\
		\begin{itemize}
			\item the set $\Phi(\mathcal{F}^0)\cap U_\delta^{1}$ is non-convex;
			\item for each irrational $\al$, the set $\Phi(\mathcal{F}^\infty_\al)\cap U_\delta^{2-\eps}$ is non-convex;
			\item for each irrational $\al$, the set $\Phi(\mathcal{F}^0_\al)\cap U_\delta^{1}$ is non star-shaped with respect to $\phi\equiv 0\in U_\delta^{1}$.
		\end{itemize}
\end{Proposition}
		This suggests that the set of perturbations in $U_\delta^{r}$ (for $r<2$) that guarantee the existence of invariant graphs has a poor geometric structure, which poses further difficulties for addressing questions (I)-(IV).

			\vspace{1ex}
		\subsection{On Questions (I) and (II)}
		By using certain modifications of  the Arnold family (see Section~\ref{2.2}), we obtain the following result.

\begin{Theorem}\label{M0}
Fix \( 0 < \varepsilon \ll 1 \) and $\iota \in \mathbb{N}$ with $\iota \geq 2$. For any rotation number \( \alpha \in (0,1) \) and any \( \delta > 0 \),
\[
\left(U_\delta^{\iota - \varepsilon} \setminus U_\delta^\iota\right) \cap \Phi(\mathcal{F}_\alpha^\omega) \neq \emptyset.
\]
\end{Theorem}

\vspace{1ex}

 Theorem~\ref{M0} is proved by explicitly constructing an element in the corresponding set. More precisely, we show that there exists a perturbation \(\phi \in C^\omega(\mathbb{T})\) whose \(C^{\iota-\varepsilon}\)-norm is arbitrarily small but whose \(C^\iota\)-norm is large, such that for every rotation number \(\alpha\), one can find \(\beta\) (depending on \(\alpha\)) for which the twist map \(f_\beta^\phi\) defined by \eqref{Fpar} admits a $C^\omega$ invariant graph with rotation number $\alpha$.

	\vspace{1ex}

By \eqref{hrfo}, the circle map \(g_\beta\) corresponding to \(f_\beta^\phi\) is also real-analytic. According to the Denjoy-Herman-Yoccoz theory of circle maps \cite{De,H11,Y1,Y2,Y3}, we have the following:
\begin{itemize}
 \item If \(\alpha\) is rational, then \(g_\beta\) is not conjugate to \(R_\alpha\).
    \item If \(\alpha\) is irrational, then \(g_\beta\) is at least \(C^0\)-conjugate to \(R_\alpha\).
    \item If \(\alpha \in \mathcal{H}\), then \(g_\beta\) is \(C^\omega\)-conjugate to \(R_\alpha\) (i.e., \(C^\omega\)-linearizable); see Remark~\ref{hbdd} for further details on the set \(\mathcal{H}\).
\end{itemize}

\begin{Remark}\label{hbdd}
The set \(\mathcal{H}\) is larger than the Diophantine numbers but smaller than the Brjuno numbers; see \cite{Y3} for its precise definition. In terms of the continued fraction expansion, let \(\{p_n / q_n\}_{n \in \mathbb{N}}\) be the sequence of convergents of \(\alpha\). Then \(\alpha \in \mathcal{H}\) if there exists \(\mu > 0\) such that
\[
\ln(q_{n+1}) \leq C (\ln q_n)^\mu \quad \text{for all } n.
\]
If \(\{a_n\}\) are the partial quotients in the continued fraction expansion of \(\alpha\), then any \(\alpha\) satisfying
\[
e^{(a_n)^\mu} \leq a_{n+1} \leq e^{a_n} \quad \text{for all } n,
\]
for some \(0 < \mu < 1\), is a Brjuno number but does not belong to \(\mathcal{H}\).
\end{Remark}
For \(\alpha \notin \mathcal{H}\), Yoccoz constructed examples of real-analytic diffeomorphisms with rotation number \(\alpha\) that are not analytically linearizable. Theorem~\ref{M0} implies that under perturbations satisfying \(\phi \in U_\delta^{\iota - \varepsilon} \setminus U_\delta^\iota\), invariant graphs with rotation number \(\alpha \notin \mathcal{H}\) can still be preserved. This leads to a natural question: how pathological is the dynamics on these preserved invariant graphs? Building on Herman's work, we provide a partial answer in the generic case.

We now recall the notion of singular conjugation.
\begin{Definition}
Let \( f \in D^1(\mathbb{T}) \) with \( \rho(f) = \alpha \in \mathbb{R} \setminus \mathbb{Q} \) such that \( f \) is \( C^0 \)-conjugate to \( R_\alpha \); write \( f = h^{-1} \circ R_\alpha \circ h \). We say that \( f \) is singular conjugate to \( R_\alpha \) if \( h \) is not absolutely continuous (i.e., there exists a Borel set \( \mathcal{B} \) such that \( m(\mathcal{B}) = 0 \) but \( m(h(\mathcal{B})) \neq 0 \)).
\end{Definition}

Define \( C^{a.c.}(\mathbb{T}) \) as the space of absolutely continuous, \(\mathbb{Z}\)-periodic functions.
Define the space
\[
\mathcal{O}_\alpha^{k,r} := \mathcal{O}_\alpha^r \cap \mathcal{F}_\alpha^k.
\]

\begin{Theorem}\label{M01}
Fix \( 0 < \varepsilon \ll 1 \) and $\iota \in \mathbb{N}$ with $\iota \geq 2$. There exists a residual set \( \mathcal{R} \subset [0,1] \) such that for every \( \alpha \in \mathcal{R} \) and every \( \delta > 0 \),
\[
\left(U_\delta^{\iota - \varepsilon} \setminus U_\delta^\iota\right) \cap \left(\Phi(\mathcal{O}_\alpha^{\omega,0}) \setminus \Phi(\mathcal{O}_\alpha^{\omega,\mathrm{a.c.}})\right) \neq \emptyset.
\]
\end{Theorem}

\vspace{1ex}

 By the Baire category theorem, the set \(\mathcal{R}\) is dense in \([0,1]\). To prove Theorem~\ref{M01}, we construct a perturbation \(\phi \in C^\omega(\mathbb{T})\) with the required properties such that for every rotation number \(\alpha \in \mathcal{R}\), there exists \(\beta\) (depending on \(\alpha\)) for which the twist map \(f_\beta^\phi\) defined by \eqref{Fpar} admits a \(C^\omega\) invariant graph on which the dynamics is singular conjugate to the rigid rotation \(R_\alpha\).

 \begin{itemize}
\item For $\iota = 2$, Theorem~\ref{M0} implies that the preimage under $\Phi$ of the set $(U_\delta^{2 - \varepsilon} \setminus U_\delta^2) \cap \Phi(\mathcal{F}^0)$ can contain invariant graphs with arbitrary rotation numbers, whether rational or irrational. (Note that if $\alpha$ is rational, the restriction of $\Phi$ to $\mathcal{F}^0_{\alpha}$ is not necessarily invertible.)

\item For $\iota = 3$, Theorems~\ref{M0} and~\ref{M01} provide answers to Questions (I) and (II) above. Specifically, there are no restrictions on the rotation numbers of the invariant graphs, and the dynamics on these graphs exhibit diversity related to the arithmetic properties of the rotation numbers.
 \end{itemize}

\begin{Remark}
According to \cite{Y3}, the set $\mathcal{H}$ is of type $F_{\sigma\delta}$ but not $F_\sigma$, which implies that
\[
\mathcal{Y} := [0,1] \setminus (\mathbb{Q} \cup \mathcal{H} \cup \mathcal{R}) \neq \emptyset.
\]
For $\alpha \in \mathcal{Y}$, we only know that there exists $\beta$ (depending on $\alpha$) such that the twist map $f_\beta^\phi$ defined by \eqref{Fpar} admits a $C^\omega$ invariant graph on which the dynamics is not $C^\omega$-conjugate to the rigid rotation $R_\alpha$, while it remains undetermined whether the conjugacy is singular.
\end{Remark}

\begin{Remark}
Based on Herman's formula \eqref{hrfo}, we can transform the construction of twist maps into the construction of circle maps. For certain Liouville rotation numbers, examples of $C^\omega$ circle maps that are singularly conjugate to rigid rotations are known (see, e.g., \cite[Theorem 12.5.1]{KH}). However, to the best of our knowledge, the perturbation in these known examples has large $C^0$-norm. Alternatively, using the Anosov-Katok method (i.e., approximation by conjugation), one can construct $C^\infty$ circle maps that are singularly conjugate to rigid rotations for certain Liouville numbers, with perturbations arbitrarily small in the $C^\infty$ topology (see, e.g., \cite[Theorem 12.6.1]{KH}). Nevertheless, this method is problematic to apply to $C^\omega$ constructions.
\end{Remark}

In the proofs of Theorems~\ref{M0} and~\ref{M01}, the constructions based on the modified Arnold family yield invariant graphs that are themselves $C^\omega$. In fact, for the construction given by Herman's formula \eqref{hrfo}, the regularity of the invariant graph coincides with that of the circle map $g$. The existence and breakdown of invariant graphs with lower regularity (e.g., non-$C^\infty$) have attracted considerable interest, corresponding precisely to Questions (III) and (IV). In the next section, we address these two questions.

	\vspace{1ex}

		\subsection{{On Questions (III) and (IV)}}
		These questions are related to two problems posed by Mather in 1998. Specifically, Mather asked the following questions (\cite[Page 181-182]{MY}):

		\vspace{1ex}
		
		\begin{itemize}
			\item \textbf{Question 1:} {\it Given a \( C^\infty \) twist diffeomorphism and an invariant graph which is not \( C^\infty \), is it possible to destroy it by an arbitrarily small \( C^\infty \) perturbation?}
		\end{itemize}
		
		\vspace{1ex}

		\begin{itemize}
			\item \textbf{Question 2:} {\it Does there exist an example of a $C^r$  area-preserving twist map with an invariant graph which is not $C^1$ and that contains no periodic point? (separate the
				question for each $r\in [1,+\infty]\cup \{\omega\}$)}
		\end{itemize}
		
		\vspace{1ex}

		\subsubsection{A Positive Evidence for Question 1}
		
		The first question concerns the $C^\infty$ destruction of invariant graphs. Mather~\cite{M4} showed that any invariant graph with a Liouville rotation number can be destroyed by arbitrarily small \( C^\infty \) perturbations. Later, Forni~\cite{F} proved that even real-analytic (\( C^\omega \)) perturbations can destroy such graphs for a certain subclass of frequencies (a proper subset of non-Brjuno numbers). Here, we provide a positive evidence to this question for invariant graphs with Diophantine rotation numbers.
		
		\begin{Theorem}\label{M1}
			Given \( 0 < \varepsilon \ll 1 \) and a \((D, \tau)\)-Diophantine number $\alpha$, there exists a sequence of \( C^\infty \) perturbations \( \{\phi_n\}_{n \in \mathbb{N}} \) such that
			\[
			\|\phi_n\|_{C^{3+\tau-\varepsilon}} \to 0, \quad \text{but} \quad \|\phi_n\|_{C^{3+\tau+\varepsilon}} \nrightarrow 0,
			\]
			and the following properties hold:
			\begin{itemize}
				\item The perturbed map \( F^{\phi_n} \) admits an invariant graph \( {\Gamma}_n \) with rotation number \( \alpha \);
				\item The graph \( {\Gamma}_n \) is not of class \( C^\infty \), and it can be destroyed by an arbitrarily small \( C^\infty \) perturbation.
			\end{itemize}
		\end{Theorem}
		
		\vspace{1ex}
		\begin{Remark}
			Theorem~\ref{M1} only offers a partial resolution to Mather's question, demonstrating the existence of a class of \( C^\infty \) twist diffeomorphisms for which certain non-\( C^\infty \) invariant graphs are not stable under arbitrarily small \( C^\infty \) perturbations. Whether this phenomenon holds for all such diffeomorphisms remains an open question.
		\end{Remark}
		
		\subsubsection{Partial Support for Question 2}
		
		The second question concerns the minimal attainable regularity of invariant graphs in higher regularity systems. Recent contributions to this problem appear in~\cite{Arna2,Arna14,AF}. Arnaud, in~\cite{Arna2,Arna14}, constructed explicit $C^1$ and $C^2$ maps possessing invariant graphs that include non-differentiable points, with the minimal invariant set forming a Cantor set. In a different direction, Avila and Fayad exhibited a $C^1$ map for which the minimal invariant set is the entire circle. The  question of whether invariant graphs containing non-differentiable points can arise in maps of higher regularity was subsequently emphasized by Fayad and Krikorian at the ICM 2018 (see~\cite[Question 26]{FKO}). Inspired by Mather's problem from 1998, one can pose the following variant:
		
		\vspace{1ex}
		\begin{itemize}
			\item \textbf{Question 2':} \textit{Given a \( C^\infty \) area-preserving twist map that admits an invariant curve with an irrational rotation number, which is represented as the graph of a function \( \psi \), what is the {minimal regularity} of \( \psi \)?}
		\end{itemize}
		
		According to \cite{Sa}, if an invariant graph with constant type rotation number is of class \( C^r \) with \( r > 4 \), then it must be \( C^\infty \). In \cite{KO}, Katznelson and Ornstein showed that if \( f \) is a \( C^{3+\gamma} \) area-preserving surface diffeomorphism that admits a \( C^{2+\varepsilon} \) invariant graph \( \psi \) with a ``good'' rotation number, then \( \psi \) is in fact \( C^{2+\gamma'} \) for all \( \gamma' < \gamma \). Here, a ``good'' rotation number \( \alpha \) refers to one whose continued fraction expansion \( [a_1, a_2, \ldots] \) satisfies \( a_n = \mathcal{O}(n^2) \), which is a broader condition than being of constant type.
		
		Combining these results, (\ref{hkam}) and  Theorem \ref{M1}, we obtain the following:
		
		\begin{Theorem}\label{neev}
			Given \( 0 < \varepsilon \ll 1 \) and a constant type number \( \alpha \), there exists a sequence of \( C^\infty \) perturbations \( \{\phi_n\}_{n \in \mathbb{N}} \) such that
			\[
			\|\phi_n\|_{C^{3-\varepsilon}} \to 0, \quad \text{but} \quad \|\phi_n\|_{C^{3}} \nrightarrow 0,
			\]
			and the following holds:
			\begin{itemize}
				\item The perturbed map \( F^{\phi_n} \) admits an invariant graph \( {\Gamma}_n \) with rotation number \( \alpha \);
				\item The graph \( {\Gamma}_n \) is not of class \( C^{2+\varepsilon} \) for any \( \varepsilon > 0 \).
			\end{itemize}
		\end{Theorem}

\vspace{1ex}

		In light of Theorems~\ref{M0} and~\ref{M01}, Theorem~\ref{neev} can be reformulated as follows: for any $\nu > 0$,
\begin{equation}\label{t4re}
\left(U_\delta^{3 - \varepsilon} \setminus U_\delta^3\right) \cap \left(\Phi(\mathcal{F}_\alpha^{0}) \setminus \Phi(\mathcal{F}_\alpha^{2+\nu})\right) \neq \emptyset.
\end{equation}
		\begin{Remark}\label{rearre}
			Whether Theorem~\ref{neev} remains valid under real-analytic (\( C^\omega \)) perturbations is still unknown. The major difficulty lies in the fact that \( C^\omega(\mathbb{T}) \) is not a Fr\'{e}chet space, preventing the direct application of the inverse function theorem (see \cite[Theorem 2.6.1]{H1}) to verify openness properties.
		\end{Remark}
		
		\begin{Remark}
			Numerical evidence provided by Olvera and Petrov~\cite{OP} suggests that for certain area-preserving twist maps (such as the standard map), the critical invariant circle (with rotation number equal to the golden mean \( (\sqrt{5}-1)/2 \)) is a \( C^{1+\eta} \) graph, with \( \eta \in (0.7, 1) \). The dynamics on this graph is conjugate to the rigid rotation by a \( C^{\eta'} \) function, where \( \eta' \in (0.7, 1) \). This numerical result underscores the subtlety in attempting to further lower the regularity threshold in Theorem~\ref{neev}.
		\end{Remark}
		
		\vspace{1ex}
		
		\subsection{More on the Set $U_\delta^{3-\eps} \setminus U_\delta^3$ }
Herman \cite[Theorem 4.9]{H1} proved that for each $\alpha \in \mathcal{C}$, there exists a $C^\infty$ perturbation $\phi$ that is arbitrarily small in the $C^{3-\varepsilon}$ topology such that the invariant graph with rotation number $\alpha$ of the corresponding twist map $F^\phi$ is destroyed. In the notation of \eqref{hckam1}, this result can be stated as: for each $\delta > 0$,
\begin{equation}\label{hckam2}
(U_\delta^{3 - \varepsilon} \setminus U_\delta^3) \cap \Phi(\mathcal{F}^0_\alpha) \subsetneqq U_\delta^{3 - \varepsilon} \setminus U_\delta^3.
\end{equation}
Combining the preceding results, we obtain the following:

\begin{Theorem}
Let $\alpha \in \mathcal{C}$. For each $\delta > 0$, the following strict inclusions hold:
\[
U_\delta^{3} \subsetneqq U_\delta^{3 - \varepsilon} \cap \Phi(\mathcal{O}_\alpha^\infty) \subsetneqq U_\delta^{3 - \varepsilon} \cap \Phi(\mathcal{F}_\alpha^0) \subsetneqq U_\delta^{3 - \varepsilon}.
\]
\end{Theorem}

\vspace{1ex}

The first strict inclusion follows from \eqref{hkam} and Theorem~\ref{M0}; the second from (\ref{t4re}); and the last from \eqref{hckam2}. If the rotation number is not required to be of Diophantine type, then by Theorem~\ref{M01}, one obtains invariant graphs exhibiting a wider variety of dynamical behaviors.
	
\begin{Remark}
Analogously to the set $U_\delta^r$, we define
\[
V^r_\delta := \left\{ \phi \in C^\omega(\mathbb{T}) \ \middle| \ \|\phi\|_{C^{r}} < \delta \right\}.
\]
It is known \cite{H11,H1} that $\mathcal{F}_\alpha^\omega = \mathcal{O}_\alpha^\omega$, and for each $\alpha \in \mathcal{C}$, there exists $\delta > 0$ such that
\begin{equation}\label{hkamx}
V_\delta^3 \subseteq \Phi(\mathcal{F}_\alpha^\omega).
\end{equation}
Moreover, Herman's result \cite[Theorem 4.9]{H1} also applies in the $C^\omega$ category (see \cite{W22}). Consequently, we have
\[
V_\delta^{3} \subsetneqq V_\delta^{3 - \varepsilon} \cap \Phi(\mathcal{O}_\alpha^\omega), \quad
V_\delta^{3 - \varepsilon} \cap \Phi(\mathcal{F}_\alpha^0) \subsetneqq V_\delta^{3 - \varepsilon}.
\]
However, for the reasons outlined in Remark~\ref{rearre}, it remains unclear whether the two intermediate sets coincide.
\end{Remark}

	\vspace{1ex}

In addition to the open problems mentioned above, the authors believe that answering the following question is crucial for deepening our understanding of non-KAM invariant graphs:

\begin{itemize}
    \item \textbf{Question $\diamondsuit$:} \textit{Does there exist a $C^\omega$ area-preserving twist map, which is a $C^r$-small perturbation ($r < 3$) of an integrable system, that admits an invariant graph with a constant type rotation number $\alpha$, and on which the dynamics is singularly conjugate to the rigid rotation $R_\alpha$?}
\end{itemize}

By the correspondence between twist maps and circle maps, if such a map exists in Question $\diamondsuit$, then the regularity of the preserved invariant graph cannot exceed $C^2$.

		%\vspace{1em}
%		\subsection*{Organization of this Note}
%		The note is organized as follows.
	 \vspace{2em}

 \noindent\textbf{Acknowledgement.}
 This work was partially supported by the National Natural Science Foundation of China (Grant No.~12122109).

\vspace{1ex}

\noindent\textbf{Data Availability Statement.}
		The authors state that this manuscript has no associated data and there is no conflict of interest.

	\vspace{1em}	
			\section{\sc Preliminaries}

%\begin{Proposition}\label{SSH}
%For all $0<\varepsilon<1$, $\delta>0$, $M>0$, and $\alpha\in[0,1)$, there exists a twist map $F(x,y) = (x + \beta + y + \phi(x), y + \phi(x))$ such that $\|\phi\|_{C^{3-\varepsilon}} < \delta$, $\|\phi\|_{C^3} > M$, and $F$ admits an invariant circle with rotation number $\alpha$.
%\end{Proposition}

\subsection{The Herman Formula}Inspired by Herman \cite{H1}, the proof proceeds by exploiting the relationship between  twist maps and orientation-preserving circle diffeomorphisms. We first construct a suitable circle map and then derive the corresponding twist map.

To build the required example, we consider a family of area-preserving twist maps defined by:
\begin{equation}\label{deF}
F_\beta^\phi : \mathbb{T} \times \mathbb{R} \to \mathbb{T} \times \mathbb{R}, \quad\mathrm{via}\ \ (x, y) \mapsto \left(x + \beta + y + \phi(x),\ y + \phi(x)\right),
\end{equation}
where $\phi : \mathbb{T} \to \mathbb{R}$ is a $C^1$ function with zero mean, i.e., $\int_{\mathbb{T}} \phi(\theta)\,d\theta = 0$. Recall the convention that we use the same notation for the following objects and their lifts: $\phi$, $\Gamma$, $\Psi$, and $g$.

\begin{Lemma}\label{herob}
The graph $\Gamma = \{(x, \Psi(x)) \mid x \in \mathbb{T}\}$ is an invariant circle of $F_\beta^\phi$ if and only if the following conditions are satisfied:
\begin{enumerate}
    \item[(1)] The map $g_\beta := \mathrm{Id} + \beta + \Psi + \phi$ belongs to $D^0(\mathbb{T})$;
    \item[(2)] The identity $\mathrm{Id} + \frac{1}{2}\phi = \frac{1}{2}\left(g_\beta + g_\beta^{-1}\right)$ holds.
\end{enumerate}
\end{Lemma}

\begin{proof}We denote $F:=F_\beta^\phi$ and $g:=g_\beta$ for simplicity.

($\Rightarrow$) Suppose $\Gamma$ is invariant under $F$.
\begin{enumerate}
    \item[(1)] Let $\pi$ be the projection onto the first coordinate. Then $g(x) = \pi \circ F(x, \Psi(x)) = x + \beta + \Psi(x) + \phi(x)$, which implies $g_\beta\in D^1(\T)$.
    \item[(2)] From (\ref{deF}), we have:
    \[
    F(x, \Psi(x)) = (x + \beta + \Psi(x) + \phi(x),\ \Psi(x) + \phi(x)).
    \]
    Note that $g(x) = x + \beta + \Psi(x) + \phi(x)$. The invariance of $\Gamma$ implies $\Psi(x) + \phi(x) = \Psi(g(x))$, or equivalently,
    \[
    g(x) - x - \beta = \Psi(g(x)) \quad \text{and} \quad x - g^{-1}(x) - \beta = \Psi(x).
    \]
    Using $\Psi(x) = g(x) - x - \beta - \phi(x)$, we obtain:
    \[
    x - g^{-1}(x) - \beta = g(x) - x - \beta - \phi(x),
    \]
    which rearranges to:
    \[
    \mathrm{Id} + \frac{1}{2}{\phi} = \frac{1}{2}\left({g} + {g}^{-1}\right).
    \]
\end{enumerate}

($\Leftarrow$) Conversely, assume (1) and (2) hold. Then:
\[
F(x, \Psi(x)) = (g(x),\ \Psi(x) + \phi(x)).
\]
From condition (2), we have $x - g^{-1}(x) - \beta = \Psi(x)$. Hence,
\[
\Psi(x) + \phi(x) = \Psi(g(x)).
\]
Since $g\in D^1(\T)$, it follows that $\Gamma$ is an invariant graph of $F$.
\end{proof}

\vspace{1ex}

\subsection{The Arnold Family}\label{2.2}

The Arnold family (or standard family) is a two-parameter family of circle maps defined by:
\begin{equation}\label{arnf}
f_{\lambda, \sigma}(x): = x + \lambda + \frac{\sigma}{2\pi} \sin(2\pi x) \pmod{1}, \quad \lambda, \sigma \in \mathbb{R}.
\end{equation}
This family was first studied by V. Arnold \cite{Ar}. In this work, we consider the parameter range $\sigma \in [0,1)$, which ensures that each map is a diffeomorphism.

\begin{itemize}
\item When $\sigma = 0$, the map reduces to a rigid rotation: $f_{\lambda,0}(\theta) = R_\lambda(\theta)$, so $\rho(f_{\lambda,0}) = \lambda$.
\item For $\lambda = 0$, the point $\theta = 0$ is a fixed point for all $\sigma$, so $\rho(f_{0,\sigma}) = 0$.
\item The rotation number $\rho(f_{\lambda,\sigma})$ is continuous in both $\lambda$ and $\sigma$, and  $\rho(f_{\lambda,\sigma})$ takes all values in $[0,1]$ as $\lambda$ varies.
    \end{itemize}
 Let us recall a key fact about the Arnold family (see \cite[Chapter I, Lemma 4.2]{de} for instance).

\begin{Lemma}
Fix $0 < \sigma < 1$. Suppose $\lambda_0$ is such that $\rho(f_{\lambda_0,\sigma}) = p/q$ is rational. Then there exists an interval $I$ containing $\lambda_0$ with nonempty interior such that $\rho(f_{\lambda,\sigma}) = p/q$ for all $\lambda \in I$. Moreover, the function $\lambda \mapsto \rho(f_{\lambda,\sigma})$ is increasing, locally constant at rational rotation numbers, and strictly increasing at irrational values.
\end{Lemma}

\begin{figure}[htbp]\label{fi31}
\small \centering
\includegraphics[width=8.5cm]{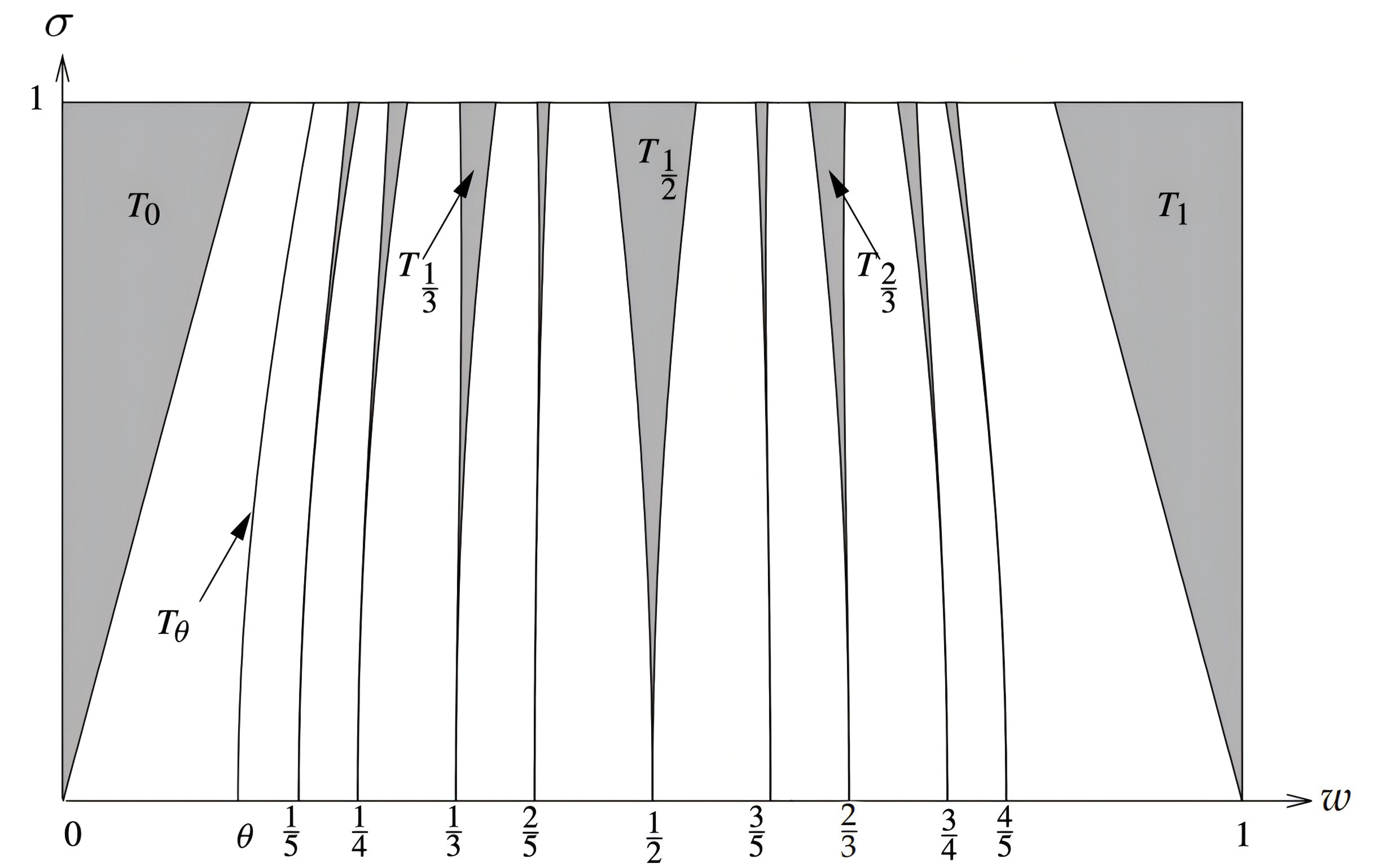}
\caption{Arnold tongues: level sets $T_\rho$ of the rotation number $\rho$ for $w,\sigma\in[0,1]$. Tongues of every rational number are connected and have nonempty interior.  Besides, if $\theta$ is an irrational number, $T_\theta$ may not have an interior point.}
\end{figure}

The following result was proved by Herman \cite[Page 169, Prop. 1.13]{H11}.
\begin{Lemma}\label{Yha}
Let $0 < \sigma < 1$. Define
\[
K_\sigma := [0,1] \setminus \mathrm{Int}\{\lambda \mid \rho(f_{\lambda,\sigma}) \in \mathbb{Q} \cap [0,1]\}.
\]
There exists a residual set in $K_\sigma$ such that for each $\lambda \in K_\sigma$, the map $f_{\lambda,\sigma}$ is singular conjugate to $R_{\alpha}$ with $\alpha = \rho(f_{\lambda,\sigma})$.
\end{Lemma}

We now consider a modification of the classical Arnold family. For $\kappa > 0$, choose $n$ sufficiently large such that $n^{\kappa} > 2\pi$, and define
\begin{equation}\label{armod}
f_{\beta,n}(x) := x + n\beta + \frac{1}{n^{\kappa}} \sin(2\pi x).
\end{equation}

Let
\[
h_n(\beta) := \rho(f_{\beta,n}).
\]
For $a \in \mathbb{R} \setminus \{0\}$ and a set $U$, define
\[
aU := \{a u \mid u \in U\}, \quad a + U := \{a + u \mid u \in U\}.
\]
Note that affine maps are homeomorphisms. The following is immediate:

\begin{Lemma}\label{resll}
For each $a \in \mathbb{R} \setminus \{0\}$, if $U$ is a residual set in $V$, then $aU$ (resp. $a + U$) is also a residual set in $aV$ (resp. $a + V$).
\end{Lemma}

By Lemma~\ref{Yha}, a direct translation yields:

\begin{Lemma}\label{Yhax}
Let $N \in \mathbb{N}$ satisfy $N^{\kappa} > 2\pi$. Define
\[
K_N := [0,1] \setminus \mathrm{Int}\{N\beta \mid h_N(\beta) \in \mathbb{Q} \cap [0,1]\}.
\]
There exists a residual set $G_N$ in $K_N$ such that for each $\beta \in \frac{1}{N} G_N$, the map $f_{\beta,N}$ is singular conjugate to $R_{\alpha}$ with $\alpha = h_N(\beta)$.
\end{Lemma}

For $N \in \mathbb{N}$ with $N^{\kappa} > 2\pi$, define
\[
\mathcal{K}_N := [0, N] \setminus \mathrm{Int}\{N\beta \mid h_N(\beta) \in \mathbb{Q} \cap [0,N]\},
\]
\[
\mathcal{D}_N := \mathcal{K}_N \cap \{N\beta \mid h_N(\beta) \in \mathbb{Q} \cap [0,N]\}.
\]
Then the following hold:
\begin{itemize}
\item $h_N: \mathcal{K}_N \to [0, N]$ is surjective;
\item $h_N: \mathcal{K}_N \setminus \mathcal{D}_N \to [0, N] \setminus \mathbb{Q}$ is injective.
\end{itemize}

By \cite[Page 38, Prop. 5.3(b)]{H11}, we have:

\begin{Lemma}\label{Yhamdx}
For each residual set $\mathcal{R} \subset \mathcal{K}_N$, the image $h_N(\mathcal{R})$ is also a residual set in $[0, N]$.
\end{Lemma}

To prove Theorem~\ref{M0}, we require a further modification of the Arnold family. The following lemma relates the rotation numbers of the modified system to those of $f_{\beta,n}$. See \cite[Chapter XI.2]{H11} for more general argument.

\begin{Lemma}\label{Yro}
Fix $\ka>0$. Given $\beta\in\R$, let
\[
g_{n}(x): = x + \beta + \frac{1}{n^{\ka+1}} \sin(2\pi n x).
\]
Then  \( \rho(g_{n}) = \frac{1}{n} \rho(f_{\beta,n}) \), where $\rho(f_{\beta,n}) $ is defined by (\ref{armod}).
\end{Lemma}

\begin{proof}

	Fix \( x_0 \in (0,1) \) and consider the orbit \( \{x_m\}_{m=0}^{\infty} \) where \( x_m = g_{n}^m(x_0) \). Then
\[
x_{m+1} = x_m + \beta + \frac{1}{n^{\ka+1}} \sin(2\pi n x_m),
\]
which implies
\[
n x_{m+1} = n x_m + n\beta + \frac{1}{n^{\ka}} \sin(2\pi n x_m).
\]
Let \( y_m = n x_m \). Then \( \rho(g_{n}) = \lim_{m\to\infty} \frac{x_m}{m} = \frac{1}{n} \lim_{m\to\infty} \frac{y_m}{m} \), where \( \{y_m\} \) satisfies
\[
y_{m+1} = y_m + n\beta + \frac{1}{n^{\ka}} \sin(2\pi y_m).
\]
Define
\begin{equation}\label{fbex}
f_{n}(x): = x + n\beta + \frac{1}{n^{\ka}} \sin(2\pi x),
\end{equation}
which implies
\( \rho(f_{n}) = n \rho(g_{n}) \).	
	\end{proof}

\section{\sc Proof of Theorem \ref{M0}}
		We prove Theorem \ref{M0} by establishing the following  proposition.

\begin{Proposition}\label{SSR}
Let $\iota\in\N$ with $\iota\geq 2$. Fix $0<\varepsilon\ll 1$. For each  $\delta>0$, $M>0$, and $\alpha\in(0,1)$, there exists a twist map $F_\beta^\phi(x,y) = (x + \beta + y + \phi(x), y + \phi(x))$ such that $\|\phi\|_{C^{\iota-\varepsilon}} < \delta$, $\|\phi\|_{C^\iota} > M$, and $F_\beta^\phi$ admits an invariant circle with rotation number $\alpha$.
\end{Proposition}

\vspace{1ex}

According to \emph{Fa\`{a} di Bruno's formula}, the $n$-th derivative of a composite function is given by:
\begin{align*}
\frac{d^n}{dx^n} f(g(x)) = \sum_{\sum_{k=1}^n k m_k = n} \frac{n!}{m_1!\,m_2!\,\cdots\,m_n!}
\cdot f^{(m_1 + \cdots + m_n)}(g(x))
\cdot \prod_{k=1}^n \left(\frac{g^{(k)}(x)}{k!}\right)^{m_k}.
\end{align*}

Fix \( 0 < \varepsilon < \frac{1}{2} \) and choose \( 0 < \nu < \varepsilon \). Define
\[
g_{\beta,n}(x) := x + \beta + \frac{1}{n^{\iota-\nu}} \sin(2\pi n x),
\]
\[
\phi_{\beta,n}(x) := g_{\beta,n}(x) + g_{\beta,n}^{-1}(x) - 2x.
\]
Then we obtain:
\[
\phi_{\beta,n}(x) = \frac{1}{n^{\iota-\nu}}\left[\sin(2\pi n x) - \sin\left(2\pi n g_{\beta,n}^{-1}(x)\right)\right].
\]
We will prove the estimate
\[
\|\phi_{\beta,n}\|_{C^{\iota-\varepsilon}} = O\left(\frac{1}{n^{\varepsilon - \nu}}\right),
\]
so that for sufficiently large \( n \), we have \( \|\phi_{\beta,n}\|_{C^{\iota-\varepsilon}} < \delta \).

Furthermore, we will show that \( \|\phi_{\beta,n}\|_{C^\iota} \) admits a positive lower bound. This allows us to construct a twist map \( F_\beta^{\phi_{\beta,n}} \) with an invariant circle of any prescribed rotation number, such that \( \|F_\beta^{\phi_{\beta,n}} - F_\beta^0\|_{C^\iota} \) is not small, where \( F_\beta^0 \) is the integrable system \( F_\beta^0(x,y) = (x + y + \beta, y) \).

We denote $D^n:=\frac{d^n}{dx^n}$ for simplicity.  Differentiating $g_{\beta,n}$ for  $m$ times yields:
\begin{align*}
Dg_{\beta,n}(x) &= 1 + \frac{2\pi}{n^{\iota-\nu-1}}\cos(2\pi n x) = O(1), \\
D^{m}g_{\beta,n}(x) &= \frac{(2\pi)^m}{n^{\iota-\nu-m }}\sin\left(2\pi n x + \frac{m\pi}{2}\right) = O\left(\frac{1}{n^{\iota-\nu-m}}\right), \quad \forall m \in \mathbb{Z}_{\geq 2}.
\end{align*}
Since $g_{\beta,n}$ is invertible and differentiable, and
\[
x = g_{\beta,n}^{-1}(x) + \beta + \frac{1}{n^{\iota-\nu}}\sin\left(2\pi n g_{\beta,n}^{-1}(x)\right),
\]
we differentiate $g_{\beta,n}^{-1}$ twice:
\begin{align*}
Dg_{\beta,n}^{-1}(x) &= \frac{1}{1 + \frac{2\pi}{n^{\iota-\nu-1}}\cos\left(2\pi n g_{\beta,n}^{-1}(x)\right)} = O(1), \\
D^{2}g_{\beta,n}^{-1}(x) &= \frac{(2\pi)^{2}}{n^{\iota-\nu-2}} \cdot \sin\left(2\pi n g_{\beta,n}^{-1}(x)\right) \cdot \left(Dg_{\beta,n}^{-1}(x)\right)^{3} = O\left(\frac{1}{n^{\iota-\nu-2}}\right).
\end{align*}
Then we compute:
\begin{align*}
D\phi_{\beta,n}(x) &= Dg_{\beta,n}(x) + Dg_{\beta,n}^{-1}(x) - 2 \\
&= \frac{2\pi}{n^{\iota-\nu-1}}\cos(2\pi n x) + \frac{1}{1 + \frac{2\pi}{n^{\iota-\nu-1}}\cos\left(2\pi n g_{\beta,n}^{-1}(x)\right)} - 1 \\
&= O\left(\frac{1}{n^{\iota-\nu-1}}\right).
\end{align*}
For \( m \geq 2 \), differentiating \( \phi_{\beta,n} \) $m$ times gives:
\begin{align*}
D^{m}\phi_{\beta,n}(x) = \frac{(2\pi)^m}{n^{\iota-\nu-m }}\sin\left(2\pi n x + \frac{m\pi}{2}\right) - \frac{1}{n^{\iota-\nu }}\frac{d^m}{dx^m}\sin\left(2\pi n g_{\beta,n}^{-1}(x)\right).
\end{align*}

By \emph{Fa\`{a} di Bruno's formula},
\[
D^m\sin\bigl(2\pi n g_{\beta,n}^{-1}(x)\bigr)
= \sum_{\sum k b_k = m} \frac{m!}{\prod b_k!}
(2\pi n)^{B} \sin\Bigl(2\pi n g_{\beta,n}^{-1}(x) + \frac{B\pi}{2}\Bigr)
\prod_{i=1}^m \left(\frac{D^{i}g_{\beta,n}^{-1}(x)}{i!}\right)^{b_i},
\]
where $B = b_1 + \cdots + b_m$.

From $g_{\beta,n}(g_{\beta,n}^{-1}(x)) = x$, for $m \geq 2$:
\[
0 = D^mg_{\beta,n}(g_{\beta,n}^{-1}(x))
= \sum_{\sum k b_k = m} \frac{m!}{\prod b_k!}
D^{B}g_{\beta,n}(g_{\beta,n}^{-1}(x))
\prod_{i=1}^m \left(\frac{D^{i}g_{\beta,n}^{-1}(x)}{i!}\right)^{b_i}.
\]
Isolating the $b_m = 1$ term:
\[
D^{m}g_{\beta,n}^{-1}(x) = -\frac{\sum_{\sum k b_k = m}^{b_m \neq 1} \frac{m!}{\prod b_k!}
D^{B}g_{\beta,n}(g_{\beta,n}^{-1}(x))
\prod_{i=1}^m \left(\frac{D^{i}g_{\beta,n}^{-1}(x)}{i!}\right)^{b_i}}{Dg_{\beta,n}(g_{\beta,n}^{-1}(x))}.
\]

By induction, for $m \geq 2$:
\begin{align*}
D^{m}g_{\beta,n}^{-1}(x) &= -\frac{(2\pi)^m}{n^{\iota-\nu-m}}\sin\Bigl(2\pi n g_{\beta,n}^{-1}(x)+\frac{m\pi}{2}\Bigr)(Dg_{\beta,n}^{-1}(x))^{m+1}
+ O\left(\frac{1}{n^{2\iota-2\nu-m-1}}\right)\\
&= O\left(\frac{1}{n^{\iota-\nu-m}}\right).
\end{align*}

Thus,
\begin{align*}
D^{m}\phi_{\beta,n}(x) &= \frac{(2\pi)^m}{n^{\iota-\nu-m}}\left[\sin\Bigl(2\pi n x + \frac{m\pi}{2}\Bigr)
- \sin\Bigl(2\pi n g_{\beta,n}^{-1}(x)+\frac{m\pi}{2}\Bigr)(Dg_{\beta,n}^{-1}(x))^m\right] \\
&\quad + O\left(\frac{1}{n^{\iota-m+1-2\nu}}\right).
\end{align*}

Using the H\"{o}lder estimate $[f]_\theta \leq 2\|f\|_\infty^{1-\theta}\|f'\|_\infty^\theta$ for $\theta \in (0,1)$:
\begin{align*}
\|\phi_{\beta,n}\|_{C^{\iota-\varepsilon}} &= \sum_{i=0}^{\iota-1}\|\phi_{\beta,n}^{(i)}\|_\infty + [\phi_{\beta,n}^{(\iota-1)}]_{1-\varepsilon} \\
&\leq O\left(\frac{1}{n^{1-\nu}}\right) + 2\|\phi_{\beta,n}^{(\iota-1)}\|_\infty^\varepsilon \|\phi_{\beta,n}^{(\iota)}\|_\infty^{1-\varepsilon} \\
&= O\left(\frac{1}{n^{1-\nu}}\right) + O\left(\frac{1}{n^{(1-\nu)\varepsilon}}\right)O\left(\frac{1}{n^{-\nu(1-\varepsilon)}}\right) \\
&= O\left(\frac{1}{n^{\varepsilon-\nu}}\right) \xrightarrow[n\to\infty]{} 0 \quad (\varepsilon > \nu).
\end{align*}
Define
\[
\Delta := \frac{\cos\left(2\pi n g_{\beta,n}^{-1}(x)\right)}{n^{\iota-\nu-1}} = O\left(\frac{1}{n^{\iota-\nu-1}}\right).
\]
Then
\begin{align*}
D^{\iota}\phi_{\beta,n}(x)
&= (2\pi)^{\iota} n^{\nu} \left[\sin\left(2\pi n x + \frac{\iota\pi}{2}\right)
- \frac{\sin\left(2\pi n g_{\beta,n}^{-1}(x) + \frac{\iota\pi}{2}\right)}{(1+\Delta)^{\iota}}\right]
+ O\left(\frac{1}{n^{1-2\nu}}\right) \\
&= (2\pi)^{\iota} n^{\nu} \left[\sin\left(2\pi n x + \frac{\iota\pi}{2}\right)
- \sin\left(2\pi n g_{\beta,n}^{-1}(x) + \frac{\iota\pi}{2}\right)\right]
+ O\left(\frac{1}{n^{1-2\nu}}\right) \\
&= (2\pi)^{\iota} n^{\nu} \cdot 2\cos\left[\pi n (x+g_{\beta,n}^{-1}(x)) + \frac{\iota\pi}{2}\right]
\sin\left[\pi n (x-g_{\beta,n}^{-1}(x))\right]
+ O\left(\frac{1}{n^{1-2\nu}}\right).
\end{align*}

Let $y := 2\pi n g_{\beta,n}^{-1}(x)$. Since
\[
x = g_{\beta,n}^{-1}(x) + \beta + \frac{1}{n^{\iota-\nu}}\sin\left(2\pi n g_{\beta,n}^{-1}(x)\right),
\]
we have
\begin{align*}
D^{\iota}\phi_{\beta,n}(x)
&= (2\pi)^{\iota} n^{\nu} \cdot 2\cos\left(y + \pi n \beta + \frac{\iota\pi}{2}\right) \sin\left(\pi n \beta\right)
+ O\left(\frac{1}{n^{1-2\nu}}\right).
\end{align*}

Therefore,
\begin{align*}
\left\|\phi_{\beta,n}\right\|_{C^{\iota}}
&= \sum_{i=0}^{\iota}\left\|\phi_{\beta,n}^{(i)}\right\|_\infty
= O\left(\frac{1}{n^{1-\nu}}\right)
+ \left\|(2\pi)^{\iota} n^{\nu} \cdot 2\cos\left(y + \pi n \beta + \frac{\iota\pi}{2}\right) \sin\left(\pi n \beta\right)
+ O\left(\frac{1}{n^{1-2\nu}}\right)\right\|_\infty \\
&= 2(2\pi)^{\iota} n^{\nu} \left|\sin(\pi n \beta)\right| + O\left(\frac{1}{n^{1-2\nu}}\right)
\succcurlyeq (2\pi)^{\iota} n^{\nu} \left|\sin(\pi n \beta)\right|.
\end{align*}

For $M > 0$, define
\begin{align*}
\Lambda_M &= \left\{ (\beta, n) \in \mathbb{R} \times \mathbb{N} : \|\phi_{\beta,n}\|_{C^{\iota}} > M \right\}, \\
\mathcal{C}_M &= \left\{ \rho(g_{\beta,n}) \bmod 1 : (\beta, n) \in \Lambda_M \right\}.
\end{align*}
where  we introduce the  asymptotic notation:
 \( A_n(x) \succcurlyeq B_n(x) \) if for every \( \delta > 0 \), there exists \( N = N(\delta) > 0 \) such that for all \( n > N \) and all \( x \in \mathbb{R} \), we have \( |A_n(x)| > |B_n(x)| - \delta \).

The following result is well known.
\begin{Lemma} \label{YRX}
For any $\gamma \in \mathbb{R} \setminus \mathbb{Q}$, the sequence $\{\{n\gamma\}\}_{n=1}^{\infty}$ is dense in $(0,1)$, where $\{x\} = x \bmod 1$.
\end{Lemma}
We now show that for every \( M > 0 \), the set \( \mathcal{C}_M \) contains all numbers in \( (0,1) \).

Consider
\[
g_{\beta,n}(x) = x + \beta + \frac{1}{n^{\iota-\nu}} \sin(2\pi n x).
\]
We analyze the rotation number \( \rho(g_{\beta,n}) \) and relate \( \|D^{\iota}\phi_{\beta,n}\|_{C^0} \) to it. Let
\begin{equation}\label{fbe}
f_{\beta,n}(x) = x + n\beta + \frac{1}{n^{\iota-\nu-1}} \sin(2\pi x).
\end{equation}

By Lemma \ref{Yro}, $\rho(f_{\beta,n}) = n\rho(g_{\beta,n})$. Note that \eqref{fbe} resembles the Arnold family \eqref{arnf} with parameters \( (\lambda, \sigma) = \left(n\beta, \frac{2\pi}{n^{\iota-\nu-1}}\right) \).

\begin{Proposition} \label{DW}
For every \( M > 0 \), we have \( \mathcal{C}_M = (0,1) \).
\end{Proposition}

\begin{proof}
For any \( \gamma \in \mathbb{R} \setminus \mathbb{Z} \), define
\[
S_\gamma = \left\{ (\lambda, \sigma) \in \mathbb{R} \times [0, 2\pi] : \rho(f_{\lambda,\sigma}) = \gamma \right\}, \quad
d_\gamma = \inf_{(\lambda,\sigma) \in S_\gamma} |\lambda|, \quad
d = \inf_{\gamma \in I_0} d_\gamma,
\]
where
\[
I_0 := \left(\frac{1}{3} - 10^{-6}, \frac{2}{3} + 10^{-6}\right).
\]
Since $\rho(f_{\lambda,\sigma})$ takes all values in $[0,1]$ as $\lambda$ varies, we have \( 0 < d < \frac{1}{2} \).

Recall the induced flat metric on \( \mathbb{T} \) is defined as
\[
\|\cdot\| := \inf_{p \in \mathbb{Z}} |\cdot + p|.
\]
For any \( \alpha \in (0,1) \), choose
\[
n > \left( \frac{2M}{(2\pi)^{\iota} |\sin(d\pi)|} \right)^{1/\nu}
\]
such that $\|n\alpha\| \in I_0$. Such an \( n \) exists:

- If \( \alpha \) is irrational, by Lemma \ref{YRX}, there exists a strictly increasing sequence \( \{s_t\} \subset \mathbb{N} \) such that \( \left| \|s_t \alpha\| - \frac{1}{2} \right| < \frac{1}{n} \).

- If \( \alpha = \frac{p}{q} \) is rational (with \( p, q \) coprime), let \( s_t = q t + k_\alpha \), where \( k_\alpha \in \{1, \dots, q\} \) minimizes
\[
\mathcal{A}(k) := \left| \left\|k \frac{p}{q}\right\| - \frac{1}{2} \right|.
\]
For large \( t \), we have
\[
s_t > \left( \frac{2M}{(2\pi)^{\iota} |\sin(d\pi)|} \right)^{1/\nu}.
\]

Let \( \gamma = n\alpha \). Since $\rho(f_{\cdot,\sigma}): \mathbb{R} \to [0,1]$ is surjective, we can choose \( \beta \) such that
\[
\left(n\beta, \frac{2\pi}{n^{\iota-\nu-1}}\right) \in S_\gamma.
\]
Then \( \rho(f_{\beta,n}) = \gamma \), so
\[
\rho(g_{\beta,n}) = \frac{\rho(f_{\beta,n})}{n} = \frac{\gamma}{n} = \alpha.
\]
Moreover,
\[
\|D^{\iota}\phi_{\beta,n}\|_{C^0} \succcurlyeq (2\pi)^{\iota} n^{\nu} |\sin(n\beta\pi)| > (2\pi)^{\iota} \cdot \frac{2M}{(2\pi)^{\iota} |\sin(d\pi)|} \cdot |\sin(d\pi)| = 2M.
\]
Hence, \( (\beta, n) \in \Lambda_M \), so \( \alpha \in \mathcal{C}_M \). Since \( \alpha \in (0,1) \) was arbitrary, we conclude \( \mathcal{C}_M = (0,1) \).
\end{proof}

This completes the proof of Proposition \ref{DW}.
	\vspace{1em}

		\section{\sc Proof of Theorem \ref{M01}}		

We now prove the theorem based on Lemmas~\ref{resll}, \ref{Yhax}, and~\ref{Yhamdx}. Fix $\kappa > 0$ and define
\begin{equation}\label{N_0}
N_0 := \min\{N \in \mathbb{N} \mid N^\kappa > 2\pi\}.
\end{equation}
Recall the modified Arnold family defined in \eqref{armod}:
\[
f_{\beta,n}(x) := x + n\beta + \frac{1}{n^{\kappa}} \sin(2\pi x).
\]
For all $N \geq N_0$, the map $f_{\beta,N}$ is a diffeomorphism.

Note that
\[
\mathcal{K}_N = \bigcup_{k=0}^{N-1} (k + K_N).
\]
Define
\[
\mathcal{G}_N := \bigcup_{k=0}^{N-1} (k + G_N).
\]
By Lemmas~\ref{resll} and~\ref{Yhax}, the set $\mathcal{G}_N$ is residual in $\mathcal{K}_N$.
Lemma~\ref{Yhax} implies that for each $\beta \in \frac{1}{N} G_N$, the map $f_{\beta,N}$ is singular conjugate to $R_\alpha$ with $\alpha = h_N(\beta)$. To complete the proof, we establish the following:

\begin{itemize}
\item \textbf{Assertion:} For each $\beta \in \frac{1}{N} \mathcal{G}_N$, the map $f_{\beta,N}$ is singular conjugate to $R_\alpha$ with $\alpha = h_N(\beta)$.
\end{itemize}

To prove this, we recall notation introduced by Herman. Let $\varphi : \mathbb{T} \to \mathbb{R}$ be measurable with respect to the Lebesgue-Haar measure $m$ on $\mathbb{T}$, and define
\[
\delta(\varphi, 0) = \int_{\mathbb{T}} \frac{|\varphi|}{1 + |\varphi|}  dm.
\]
Define the functional $\mathcal{N} : D^1(\mathbb{T}) \to \mathbb{R}_+$ by
\[
\mathcal{N}(f) = \inf_{n \geq 1} \delta\left( \frac{1}{n} \sum_{i=0}^{n-1} Df^i, 0 \right).
\]
Then $\mathcal{N}$ is upper semicontinuous in the $C^1$-topology, so $\mathcal{N}^{-1}(0)$ is a $G_\delta$ set. Let
\[
D_N := K_N \cap \{N\beta \mid h_N(\beta) \in \mathbb{Q} \cap [0,1]\}.
\]
The residual set $G_N$ in Lemma~\ref{Yhax} is given by
\[
G_N = \mathcal{N}^{-1}(0) \cap (K_N \setminus D_N).
\]

By the periodicity of $\sin(2\pi x)$, for all $i \in \mathbb{N}$ and $k \in \mathbb{N}$,
\[
Df_{\beta + \frac{k}{N}}^i = Df_{\beta}^i,
\]
which implies $\mathcal{N}(f_{\beta + \frac{k}{N}}) = \mathcal{N}(f_{\beta})$. Since
\[
h_N\left(\beta + \frac{k}{N}\right) = h_N(\beta) + k,
\]then for each $k \in \mathbb{N}$ and each $\beta \in \frac{1}{N} G_N + \frac{k}{N}$, the map $f_{\beta,N}$ is singular conjugate to $R_\alpha$ with $\alpha = h_N(\beta)$. This proves the assertion.

By Lemma~\ref{Yhamdx}, $h_N(\mathcal{G}_N)$ is residual in $[0, N]$. By Lemma~\ref{resll}, the set $\frac{1}{N} h_N(\mathcal{G}_N)$ is residual in $[0, 1]$. Define
\[
\mathcal{R} := \bigcap_{N = N_0}^{\infty} \frac{1}{N} h_N(\mathcal{G}_N),
\]
where $N_0$ is as in \eqref{N_0}. Then $\mathcal{R}$ is a residual set in $[0, 1]$. Therefore, for every $N \geq N_0$ and every $\alpha \in \mathcal{R}$, there exists $\beta = h_N^{-1}(\alpha)$ such that $f_{\beta,N}$ is singular conjugate to $R_\alpha$.

		\vspace{1em}

		\section{\sc Proof of Theorem \ref{M1}}
		The proof is inspired by Herman~\cite[Chapter II]{H1}. Given a set \( S \), denote by \( \bar{S} \) its closure and by \( \mathrm{int}(S) \) its interior in a given topology. For simplicity, we denote $U_\delta:=U_\delta^{3+\tau-\eps}$. Namely,
		\[U_\delta=\left\{ \phi \in C^\infty(\mathbb{T})\ \middle|\ \|\phi\|_{C^{3+\tau-\eps}} < \delta \right\}.\]
		This is an open and path-connected subset of \( {H}^\infty \) under the \( C^\infty \) topology. We now consider the topological space \( X = (U^r_\delta, \tau) \), where \( \tau \) is the subspace topology inherited from \( {H}^\infty \).
		Define
		\[
		W_\delta := \Phi(\mathcal{F}_\al^\infty) \cap U_\delta, \quad A_\delta := \Phi(\mathcal{F}_\al^0) \cap U_\delta.
		\]
		Since $\al$ is a Diophantine number, by Proposition \ref{Hp123}(i), \( \Phi(\mathcal{F}_\al^\infty) \) is open in \( {H}^\infty \), hence \( W_\delta \) is open in \( U_\delta \). Let \( \bar{A}_\delta \) denote the closure of \( A_\delta \) in \( U_\delta \) with respect to the \( C^\infty \) topology. Then
		\[
		W_\delta \subseteq \mathrm{int}(A_\delta) \subseteq \mathrm{int}(\bar{A}_\delta),
		\]
		and \( \mathrm{int}(\bar{A}_\delta) \neq \emptyset \). By Proposition \ref{Hp123}(ii), $\Phi(\mathcal{F}_\al^0)\cap H^\infty$ is closed in $H^\infty$ under the $C^1$ topology. We have
		\[\bar{A}_\delta\subseteq \Phi(\mathcal{F}_\al^0) \cap U_\delta.\]
		Let \( \partial\bar{A}_\delta \) denote the boundary of \( \bar{A}_\delta \). Then for each $\phi\in \partial\bar{A}_\delta$, $F^\phi$ admits an invariant graph ${\Gamma}$ with rotation number $\al$. Recall
		\[{\Gamma}:=\{(x,{\Psi}(x))\ |\ x\in\R\}.\]
		
		\begin{itemize}
			\item \textbf{Assertion 1:} \( \partial\bar{A}_\delta \neq \emptyset \).
		\end{itemize}
		Indeed, suppose the contrary, \( \partial\bar{A}_\delta = \emptyset \), then \( \bar{A}_\delta = \mathrm{int}(\bar{A}_\delta) \), implying that \( \bar{A}_\delta \) is both open and closed in \( U_\delta \). Since \( U_\delta \) is connected, this yields \( \bar{A}_\delta = U_\delta \). However, we have \( \bar{A}_\delta \neq U_\delta \) according to (\ref{hckam2}). Contradiction.
		
		\begin{itemize}
			\item \textbf{Assertion 2:} For each \( \phi \in \partial\bar{A}_\delta \) and any small neighborhood \( O_\phi \subseteq U_\delta \), there exists \( \varphi \in O_\phi \) such that \( \varphi \notin \Phi(\mathcal{F}_\al^0) \).
		\end{itemize}
		Assume otherwise: there exists \( \phi^* \in \partial\bar{A}_\delta \) and a small open neighborhood \( O_{\phi^*} \subseteq U_\delta \) with \( O_{\phi^*} \subseteq \Phi(\mathcal{F}_\al^0) \), hence \( O_{\phi^*} \subseteq A_\delta \), contradicting \( \phi^* \in \partial\bar{A}_\delta \).
		
		\begin{itemize}
			\item \textbf{Assertion 3:} The invariant graph \( {\Gamma} \) admitted by $F^\phi$ is not of class \( C^\infty \).
		\end{itemize}
		By Proposition \ref{Hp123}(iii), \( \Phi \) is injective for irrational \( \alpha \), so \( \Phi^{-1} \) is well-defined on \( \Phi(\mathcal{F}_\al^0) \). If \( \phi \in \partial\bar{A}_\delta \) were such that \( g := \Phi^{-1}(\phi) \in \mathcal{F}_\al^\infty \), then \( \phi \in W_\delta \subseteq \mathrm{int}(\bar{A}_\delta) \), contradicting \( \phi \in \partial\bar{A}_\delta \). Since \( {\Psi} = g - \mathrm{Id} - \phi \), it follows that \( {\Psi}\notin C^\infty \).
		
		\vspace{1ex}
		
		Taking \( \delta = \frac{1}{n} \), we obtain a sequence \( \{\phi_n\}_{n \in \mathbb{N}} \) with
		\[
		\phi_n \in \partial\bar{A}_{1/n}, \quad \text{so that } \|\phi_n\|_{C^{3+\tau-\varepsilon}} \to 0,\ \text{but } \|\phi_n\|_{C^{3+\tau+\eps}} \nrightarrow 0.
		\]

	\end{CJK*}

\end{document}